\documentclass[preprint,12pt]{elsarticle}
\usepackage{graphicx} 
\usepackage{amsmath,amssymb,amsfonts,amsthm}
\usepackage[english]{babel}
\usepackage{hyperref}
\newtheorem{assumption}{Assumption}
\newtheorem{remark}{Remark}
\newtheorem{definition}{Definition}
\newtheorem{proposition}{Proposition}
\newtheorem{problem}{Problem}
\newtheorem{lemma}{Lemma}
\newtheorem{theorem}{Theorem}
\usepackage{xcolor}
\newtheorem{example}{Example}
\newcommand{\be}{\begin{equation}}
	\newcommand{\ee}{\end{equation}}
\newcommand{\ben}{\begin{equation*}}
	\newcommand{\een}{\end{equation*}}
\newcommand{\mc}{\mathcal}
\newcommand{\ba}{\begin{array}}
	\newcommand{\ea}{\end{array}}
\newcommand{\R}{\mathbb{R}}

\newenvironment{contexample}{
	\addtocounter{example}{-1} \begin{example}[continued]}{
	\end{example}\addtocounter{example}{0}}

\renewcommand{\d}{\text{d}}

\begin{document}
		\begin{frontmatter}
		
		\title{On Convexity of Optimal Multi-Commodity Freeway Network Control}
		
		\author[inst1]{Davide Sipione\corref{cor1}\fnref{funding}}
		\ead{davide.sipione@polito.it}
		
		\author[inst1]{Giacomo Como\fnref{funding}}
		\ead{giacomo.como@polito.it}
		
		\author[inst2]{Gustav Nilsson}
		\ead{gustav.nilsson@inria.fr}
\address[inst1]{Department of Mathematical Sciences, Politecnico di Torino, Turin, Italy}
\address[inst2]{Univ. Grenoble Alpes, Inria, CNRS, Grenoble INP, GIPSA-Lab, 38000 Grenoble, France}
\cortext[cor1]{Corresponding author}

\begin{abstract}
	We study a multi-commodity Freeway Network Control (FNC) problem aiming at achieving optimal operation of a transportation network through the use of ramp metering and variable speed limits. Straightforward formulations of both single- and multi-commodity FNC problems based on the Cell Transmission Model are known to be non-convex, mainly due to the congestion effects at diverge junctions. However, recent studies have shown that it is possible to formulate a tight convex relaxation of the single-commodity FNC problem.
	
	We extend  these results to the multi-commodity FNC problem by considering concave commodity-specific demand functions and concave aggregate supply functions, so that different variable speed limits can be applied to different commodities. Hence, it is possible to efficiently compute the optimal control action to reduce congestion phenomena in the network. We also present a case study of a segment of the freeway network in California, using data from the PeMS database, to demonstrate the effectiveness of the proposed solution. Finally, we draw a comparison with a setting where the multi-commodity flows are modeled and controlled as a single-commodity flow, to emphasize the relevance of acting separately on different classes of vehicles.

\end{abstract}


\fntext[funding]{This work was partially supported by the Research Project PRIN 2022 “Extracting Essential Information and Dynamics from Complex Networks” (Grant Agreement number 2022MBC2EZ) funded by the Italian Ministry of University and Research.}

\end{frontmatter}

\section{Introduction}
For decades, researchers have studied control strategies to reduce congestion and enhance both societal and user-level performance in freeway networks, see~\cite{Peeta:2001}. Effective traffic management remains a crucial topic in transportation research, as travel demand grows and user heterogeneity increases. Moreover, the limited amount of existing infrastructure, combined with the need to address environmental challenges, highlights the necessity for a fundamental shift in how freeway systems are managed and optimized. 

A realistic framework to model first-order traffic dynamics is the celebrated Cell Transmission Model (CTM), first introduced in~\cite{DAG1, DAG2}, which successfully discretizes the Lighthill-Whitham and Richards (LWR) model of kinematic waves, see \cite{KynWave} and \cite{KynWave2}. In this setting, it is possible to formulate the Freeway Network Control (FNC) problem, an optimal control problem that aims at achieving optimal traffic conditions through the use of variable speed limiting--- see~\cite{VLS1,VLS2}--- and ramp metering, which involves the usage of traffic signals to limit the traffic flow entering the network~\cite{Gomes:2006,Mura:2012,Schmitt:2017}. Early attempts at solving this problem can be found in \cite{Zhang1996,Zhang1999}, where the authors investigate ramp metering on a first-order discrete-time model and numerically solve the relative non-linear optimal control problem through gradient-based methods. Moreover, in \cite{Gomes:2006} the authors propose the Asymmetric Cell Transmission Model (ACTM) and show that the non-linear ramp metering optimal control problem reduces to a linear program under suitable conditions. This study is extended to also include variable speed limits in~\cite{Mura2012}. 

Some recent studies have investigated the Freeway Network Control problem as a subproblem of the Dynamic Traffic Assignment, where dynamical routing is included in the control action, which is inherently non-convex \cite{Non-Conv}. In particular, the authors of~\cite{ConvAndRob,ConvAndRob2} introduce tight relaxations of the FNC problem to yield convex formulations for any given network, by generalizing results obtained in~\cite{Zili:2000}. Building on these results, an exact convex formulation with controlled merging junction was addressed in \cite{Schmitt2018}.
Hence, this problem can be easily solved with available convex optimization solvers, such as \textit{CVX} (\cite{cvx}, \cite{gb08}). Moreover, the FNC is of interest and well-suited for distributed optimization, which can significantly reduce computational complexity, as shown in~\cite{Distr,Distr2}. For a more comprehensive overview of the freeway network
 control problem, we refer the reader to~\cite{Siri2011}.

As mentioned before, transportation networks have an intrinsically heterogeneous nature (with different vehicles such as cars, buses, trucks sharing the infrastructure), but recently even more so by the growing adoption of autonomous cars. Hence, it is becoming crucial to assess the intrinsic multi-commodity nature of transportation networks and to model the interactions between different commodity flows in a dynamic manner. Motivated by this, several models for multi-commodity traffic flows in networks have been proposed: the work in \cite{Herty:2008} adopts a PDE-based approach, while \cite{Nilsson:2014} and \cite{joumaa2023} propose models based on the CTM. Regardless of the modeling approach, the multi-commodity setting enables a significantly more complex ---yet realistic--- behavior even in static scenarios, as proven in \cite{Leighton:1999} where the max-flow min-cut algorithm cannot compute exactly the maximum throughput. 

Moreover, many studies have been published, in which different commodities are introduced to capture different traffic flow behavior. In particular, in~\cite{SO-DTA}, a CTM-like model is used in which the population is split into controlled and selfish agents but their demand function is shared. Their objective is to minimize the total congestion in the network by solving a non-convex optimization problem through the adjoint method and a multi-start approach. Similarly, in the work~\cite{SO-DTA2,SO-DTA3,joumaa2025}, the aim is to improve the system's global state by controlling a varying subset of compliant users, through variable routing or variable speed limits.  On the other hand, in~\cite{Nillson:2018}, the existence of a Nash Equilibrium for a selfish and a cooperative fleet of vehicles seeking two different optimal policies is characterized. Moreover, in~\cite{Multi3}, a model in which the population is split into cars and trucks is derived, accounting for their different lengths and free-flow speeds. Similarly, in \cite{Multi5}, cars coexist with powered two-wheelers that can overtake in between lanes, thus being less affected by congestion phenomena.  Another attempt at multi-commodity modeling is achieved in~\cite{Multi4}, where the creeping ---overtaking--- of large vehicles is considered. 

Multi-commodity modeling efforts have also been made for higher-order models, such as in~\cite{DEO2009}, where the METANET macroscopic traffic flow model is extended to the multi-commodity case, with control achieved through variable speed limits and ramp metering using a Model Predictive Control strategy. Another modeling framework where multi-commodity modeling becomes handy is when different groups of users have access to different information. With the objective of minimizing the total congestion of the network, the authors in~\cite{Festa:2019} and~\cite{Multi2} split the population based on the amount of information they can receive from the network, so that a fraction of the users can avoid the most congested areas. However, in some of these studies, it is shown that the interaction of commodities leads again to non-convex formulations of the optimization problems.

In this paper, we introduce a dynamical multi-commodity flow network model based on the CTM, in which the traffic flow is driven by commodity-specific demand functions and shared supply functions. Unlike most works, we apply ramp metering and variable speed separately to each commodity, which remains practically feasible, as shown in~\cite{Multi_comm}. Building on the arguments made in~\cite{ConvAndRob}, we formulate a tight relaxation of the multi-commodity FNC, by assuming both concavity of the demand and supply functions and convexity of the cost function. To corroborate our mathematical results, we validate the model using real traffic data from the \textit{Caltrans Performance Measurement System (PeMS)} under several settings.

The rest of the paper is organized as follows: in Section 2 the multi-commodity model is formulated and the Multi-Commodity Freeway Network Control (MC-FNC) problem is stated. Then, in Section 3, a tight and convex relaxation of the MC-FNC is provided and solved. In Section 4 simulations with real data are performed and results discussed, while in Section 5 conclusions are drawn.  

\subsection{Notation}
The sets $\mathbb{R}$ and $\mathbb{R}_+$ represent the sets of real and non-negative numbers, respectively. Given two finite sets $\mathcal{A}$ and $\mathcal{B}$ then $\mathbb{R}^\mathcal{A}$ is the set of real numbers indexed by elements of A. Similarly, $\mathbb{R}^{\mathcal{A} \times \mathcal{B}}$ is the set indexed by the product set of $\mathcal{A}$ and $\mathcal{B}$. A directed multigraph $\mathcal{G} = (\mathcal{V},\mathcal{E})$ has a set of nodes $\mathcal{V}$ and a finite multi-set of links $\mathcal{E}$. 
The tail and head node of a link $i$ in $\mc{E}$ are denoted by $\tau_i$ and $\sigma_i$, such that the link $i = \left( \tau_i,\sigma_i\right)$ is directed from $\tau_i$ to $\sigma_i$.

\section{Freeway Network Control of a CTM-based Dynamical Multi-Commodity Flow Network Model}

In this section, our model for multi-commodity transportation networks is introduced, along with the optimal control problem of interest. 

\subsection{Dynamical Multi-commodity Flow Network Model}
A transportation network is modeled as a directed multigraph $\mc{G} = (\mc{V},\mc{E})$, where $\mc{V}$ and $\mc{E}$ denote the sets of nodes and links, respectively. Different from a graph, a multigraph allows for multiple parallel links between a node pair. In this approach, links ---referred to as \textit{cells}--- represent sections of roads. Nodes correspond to \textit{junctions} that connect different cells, and can be categorized as diverge, merge, mixed, or ordinary.  Each cell is one-directional and may have several lanes. A diverge junction has only one incoming cell and more than one outgoing, a merge cell has several incoming cells to one outgoing, while a mixed junction has several incoming and outgoing cells. An ordinary junction has only one incoming and one outgoing cell, and is introduced for spatial discretization purposes of the dynamics. Furthermore, two subsets of special cells are introduced: the sets of onramps $\mc{R} \subset \mc{E}$ and offramps $\mc{S} \subset \mc{E}$, through which exogenous flow enters and exits the network, respectively. A cell can not be both an onramp and an offramp, so $\mc{R} \cap \mc{S} = \emptyset$. To simplify the notation for identifying cells that are immidiatly connected through a node, we introduce the notation $\mc{A}$ to refer to the set of adjacent cells, i.e. $\mc{A} = \left\{(i,j) \in \mc{E} \times \mc{E} : \tau_i = \sigma_j\right\}$. 

A finite set of commodities $\mc{K}$, representing different types of classes of vehicles, share the roadspace on every cell. For every cell $i \in \mc{E}$ and $k \in \mc{K}$, the state variable $x_i^{(k)} \geq 0$ represents the traffic volume of commodity $k$ in cell $i$. Hence, the state space of the model is the non-negative orthant
\be\label{eq:statespace}
\mc{X} = \R^{\mc{E}\times\mc{K}}_+ \;.
\ee
Accordingly, the state of the system is a time-varying vector $x = x(t)$, whose entries $x_i^{(k)} = x_i^{(k)}(t)$ denote the traffic volume of commodity $k$ in a cell $i$ at time $t \geq 0$. For each $k$ in $\mc{K}$, we denote by $x^{(k)}$ in $\R_+^{\mc E}$ the vector of traffic volumes of commodity $k$ across all cells.

To model the flow propagation of each commodity in the network, two families of functions are introduced. The maximum physically possible outflow $d_i^{(k)}(x_i^{(k)})$ of a commodity $k$ from a cell $i$ is assumed to depend on the traffic volume $x_i^{(k)}$ of  that commodity in that cell through a \emph{demand function} $$d_i^{(k)}: \R_+ \rightarrow \R_+\,.$$ Moreover, the maximum possible inflow $s_i(\sum\nolimits_{k \in \mc{K}}{w_i^{(k)}x_i^{(k)}})$ into a cell $i$ is assumed to depend on a weighted sum of the traffic volumes of the different commodities in cell $i$ through an aggregate \emph{supply function}  $$s_i: \R_+ \rightarrow \R_+\,.$$ Here, $w_i^{(k)}> 0$ is a positive weight on commodity-$k$ traffic volume in cell $i$. 
\begin{assumption}\label{ass:1}
	For every cell $i$ in $\mc{E}$ and commodity $k$ in $\mc{K}$ the demand functions $d_i^{(k)}$ is Lipschitz-continuous, concave non-decreasing, and such that $d_i^{(k)}(0) = 0$. Moreover, for every cell $i$ in $\mc{E}$ the supply function $s_i$ is Lipschitz-continuous, concave, and non-increasing.
\end{assumption}

\begin{example} \label{ex:dcartrucks}
	Consider a transportation network, with two commodities $\mc K = \{\text{car}, \text{truck}\}$,  in which each cell $i$ in $\mc{E}$ has a length $L_i$. Although the infrastructure is shared by both cars and trucks, they have different free-flow speeds,  $v_{ff}^{(car)}$ and $v_{ff}^{(truck)}$, respectively. Hence, their linear demand functions read
	\be\label{eq:dmnd}
	d^{(car)}_i(x_i^{(car)}) = \frac{v_{ff}^{(car)}}{L_i}x_i^{(car)}\;, \qquad d^{(truck)}_i(x_i^{(truck)}) = \frac{v_{ff}^{(truck)}}{L_i}x_i^{(truck)}\;, 
	\ee
	for every cell $i$ in $\mc E$.
	The supply function models the available space on a cell $i$ in $\mc{E}$, hence it will depend on the current mixture of cars and trucks. To take into account their different lengths, the supply functions will have the form
	
	\be\label{eq:supply}
		s_i\left(\sum\limits_{k \in \mc{K}}l^{(k)}x_i^{(k)}\right) = \max \left(0,\frac{w_i}{L_i} \cdot \frac{\beta \cdot l_i n_{lanes}  - l^{(car)}x_i^{(car)} - l^{(truck)}x_i^{(truck)}}{\tilde{l}}\right) \,,
	\ee
	where $\tilde{l}$ is the average vehicle length, i.e., $ \tilde{l} = pl^{(car)} + (1-p)l^{(truck)}$, $p$ is the percentage of cars in the network, $w_i$ is the wave speed, $n_{lanes}$ is the number of lanes in the cell, and $\beta$ is a vector of coefficients used to account for reduced capacity of the cell, i.e. closed lanes. It should be noted that although $\tilde{l}$ will be a static estimate, this estimate will only affect the slope of the supply function, i.e., regardless of the mixture between cars and trucks, the model will not allow the cell to be filled beyond capacity.
\end{example}

The network is subject to an exogenous inflow $\lambda_i(t) \geq 0$ at each onramp~$i$ in $\mc{R}$, representing the rate at which new flow enters the network. For all other cells $i$ in $\mc{E} \setminus \mc{R}$, the exogenous inflow is identically $0$, i.e., $\lambda_i(t) \equiv 0$. For each $k$ in $\mc{K}$, we denote by $\lambda^{(k)}$ the vector of exogenous inflows of commodity~$k$ across all cells. 

Different commodities may not use the same roads or, at diverge junctions, may split according to different proportions. To model this behavior, for each commodity $k$ in $\mc{K}$, we introduce time-varying commodity-dependent routing matrices $R^{(k)}(t)$ in $\R_+^{\mc{E}\times\mc{E}}$. Their entries $0 \leq R_{ij}^{(k)}(t) \leq 1$, referred to as \textit{turning ratios}, denote the fraction of flow of commodity $k$ routed from cell~$i$ to cell $j$.
\begin{assumption}\label{ass:2}
	For each commodity $k$ in $\mc{K}$ and $t \geq 0$, the corresponding routing matrix $R^{(k)}(t)$ satisfies the following conditions:
	\begin{enumerate}
		\item[(i)] 
		\be\label{eq:sub}
		\sum\limits_{j \in \mc{E}}{R_{ij}^{(k)}}(t) = \begin{cases}
			1 & i \in \mc{E} \setminus \mc{S} \\
			0 & i \in \mc{S}
		\end{cases}\; ;
		\ee
		\item[(ii)] 
		\be\label{eq:zero}
		R_{ij}^{(k)}(t) = 0\, , \quad \forall (i,j) \in \mc{E} \times \mc{E} \setminus \mc A \;.
		\ee
	\end{enumerate}
\end{assumption}
The first part of Assumption \ref{ass:2} ensures that traffic can only leave the network through an offramp, while the second part of Assumption \ref{ass:2} implies that cells may send flow to adjacent cells only.

\subsection{Controlled FIFO Dynamical Multi-Commodity Flow Networks}

 In this formulation, control, through variable speed limits and ramp metering, is achieved via time-varying parameters $0 \leq \alpha_i^{(k)}(t) \leq 1$, which act on the demand functions by modulating their slope and limiting the rate at which vehicles of commodity $k$ leave cell $i$. With a slight abuse of notation,
\be\label{eq:alfa}
d_i^{(k)}(x_i^{(k)}) := \alpha_i^{(k)}d_i^{(k)}(x_i^{(k)})\,, \qquad \forall i \in \mc{E}, \quad k \in \mc{K}\;.
\ee
This formulation allows for assigning distinct control parameters to each commodity to achieve separate speed limits for each commodity. While this approach is less common, most existing works apply uniform control across all commodities (see \cite{SO-DTA}), it remains feasible. As discussed in~\cite{Multi_comm}, many countries in Europe implement different speed limits between heavy and light vehicles. Another potential application is the differentiation of speed limits for autonomous versus conventional vehicles, or different classes of autonomous vehicles controlled centrally.

The dynamic of the system is governed by a mass conservation law, which states that the variation in traffic volume of commodity $k$ in $\mc{K}$ in a cell $i$ in~$\mc{E}$ is equal to the difference between the total inflow to and the total outflow from that cell. Formally,

\be\label{eq:massCons}
	\dot{x}_{i}^{(k)}(t)  =   \lambda_{i}^{(k)}(t) + \sum\limits_{j \in \mathcal{E}}{f_{ji}^{(k)}(x)}  - z_{i}^{(k)}(x) \,,
\ee
where $f_{ji}^{(k)}(x)$ denotes the flow of commodity $k$ sent from cell $j$ to cell $i$ and $z_i^{(k)}(x)$ denotes the total outflow of commodity $k$ from cell $i$ as
\be\label{eq:zik}
		z_{i}^{(k)}(x)  =  \begin{cases}
		\sum\limits_{j \in \mathcal{E}}{f_{ij}^{(k)}( x)} \,, & i \in \mathcal{E} \setminus \mathcal{S} \,, \\
		\alpha_i^{(k)}d_{i}^{(k)}(x_{i}^{(k)}) \,, & i \in \mathcal{S}\,,
	\end{cases}
\ee
both expressed as functions of the current state $x$. Equation \eqref{eq:massCons} can be re-stated in a more compact form, by considering vector notation:
	\be\label{compact}
\dot{x}^{(k)}(t) = \lambda^{(k)}(t) + f^{(k)}(x) - z^{(k)}(x) \,, \qquad \forall k \in \mathcal{K}\,.
\ee

The only step left is to model the relationship between the cell-to-cell flow $f_{ij}^{(k)}$ and the state $x$. To this end, a First-In-First-Out (FIFO) allocation rule is employed. Formally,

\begin{equation}\label{eq:fifo}
	f_{ij}^{(k)}(x) = \gamma_i^F(t,x)R_{ij}^{(k)}(t)\alpha_i^{(k)}d_{i}^{(k)}(x_{i}^{(k)}) \; ,
\end{equation}
where $\gamma_{i}^{F}  =\max \{  \gamma \in [0,1] :    \eqref{eq:4.5}\}$,
\be\label{eq:4.5}
\gamma \cdot \underset{\underset{R_{ij}^{(k)}(t)>0}{j \in \mathcal{E}}}{\max} \:\sum\limits_{k \in \mathcal{K}}\sum\limits_{h \in \mathcal{E}} R_{hj}^{(k)}(t)\alpha_h^{(k)}d_{h}^{(k)}(x_{h}^{(k)}) 
\le  s_j(\sum\limits_{k \in \mathcal{K}}{w_j^{(k)}x_{j}^{(k)}})\,.\ee

Equation \eqref{eq:4.5} implies that in a diverge junction, if a downstream cell is congested, then the flow directed to every downstream cell is limited such that the turning ratios of all commodity flows from that junction will still split according to their given turning ratios.  Moreover, regarding ordinary junctions, this implies a novelty with respect to the single-commodity setting. When considering a single commodity model and ordinary junction, \eqref{eq:4.5} is equivalent to the minimum between the demand and the supply function. However, in our setting, $\gamma_i^{F}$ is needed to ensure that the remaining supply is shared in proportion to the different commodities' demands.

We are now ready to define the Controlled Dynamical Multi-Commodity Flow Network formally:
\begin{definition}
	Given a nonempty finite directed multigraph $\mathcal{G} = (\mathcal{V},\mathcal{E})$, a finite set of commodities $\mathcal{K}$, supply functions $s_i$ and demand functions $d_i^{(k)}$ satisfying Assumption~\ref{ass:1}, and routing matrices $R^{(k)}$ satisfying Assumption~\ref{ass:2}, exogenous inflow array $\lambda=(\lambda_{i}^{(k)})_{i\in\mc E,k\in\mc K}$, and control parameters $\alpha_i^{(k)}$ in $\left[0,1\right]$, a Controlled Dynamical Multi-Commodity Flow Network (C-DMFN) is a dynamical system satisfying equations \eqref{eq:alfa}--\eqref{eq:4.5}.
\end{definition}
It has already been observed in \cite{Stab_D} that this problem is well-posed without control, i.e., when $\alpha_i^{(k)} \equiv 1$ for every cell $i$ in $\mc E$ and commodity $k$  in  $\mc K$. However, the same proof extends to the controlled settings.
\begin{lemma}
	Consider a C-DMFN and define $\mc{X}$ as in \eqref{eq:statespace}. Then, there exists a map
	\ben
	\varphi : \R_+ \times \mc{X} \rightarrow \mc{X}
	\een
	such that for every $\bar{x}$ in $\mc{X}$, $x(t) = \varphi(t,\bar{x})$ is the unique solution of the C-DMFN with initial condition $x(0) = \bar{x}$.
\end{lemma}
\begin{proof}
	Notice that this is always true if the nonnegative orthant is invariant. A sufficient condition for this invariance is the assumption made on $d_i^{(k)}(x_i^{(k)})$, which is $d_i^{(k)}(x_i^{(k)}) = 0$ whenever $x_i^{(k)} = 0$.
\end{proof}

To illustrate how allowing for control can improve the performance of the C-DMFN, we extend Example~\ref{ex:dcartrucks}.
\begin{contexample}
	Consider a directed multigraph $\mc{G} = (\mc{V},\mc{E})$ with topology as in Figure~\ref{fig:grafo_esempio2}. The cells $\mc R =\{
    1\} $ are onramps, while cells $\mc S = \{7,9\}$ are offramps.
	\begin{figure}
		\centering
		\includegraphics[width=0.5\linewidth]{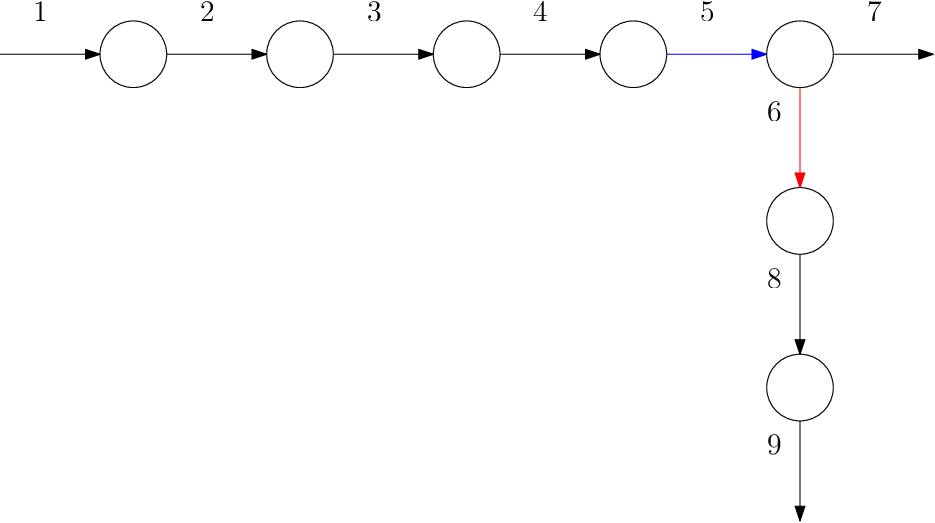}
		\caption{Network with two branches: cell $5$ (in blue) and $6$ (in red) have a reduced capacity by $\beta_5 = 0.5$ and $\beta_6 = 0.25$, respectively. Cells without a tail node are onramps, and cells without a head node are offramps.}
		\label{fig:grafo_esempio2}
	\end{figure}
	Let us consider cars and trucks as commodities that interact with each other by sharing the supply function of receiving cells. Each cell has a length $L = 0.5$ miles, while cars and trucks have the lengths of $0.0028$ and $0.0075$ miles, respectively. 
	
	Recall the demand and supply functions we previously introduced. In particular, let us set $v_{ff}^{(car)} = 60 $ mph and $v_{ff}^{(truck)} = 40 $ mph. Hence, the commodity-specific demand functions read
    $$
    d_i^{(car)}(x_i^{(car)})= 120x_i^{(car)}\,, \qquad d_i^{(truck)}(x_i^{(car)}) = 80x_i^{(car)}\,, \qquad \forall i \in \mc{E}\;.
    $$
    As far as the supply function is concerned, we set the wave speed at $w_i = 9 $ mph for every $i$ in $\mc{E}$. Moreover, the average vehicle length $\tilde{l}$ is computed assuming that $90\%$ of the vehicles in the network are cars. Then, we assume that two cells have a reduced space capacity, i.e., some construction work is being carried out on these cells. In particular, for cells $5$ and $6$, the coefficients in~\eqref{eq:supply} are set to $\beta_5 = 0.5$ and $\beta_6 = 0.25$, respectively. Lastly, the two commodities differ in their turning ratios at the diverge node. In particular, trucks are not allowed to exit the network through cell $7$, i.e., $R^\text{(truck)}_{5,7} = 0$. On the other hand, cars at the diverge junction are equally split, i.e., $R^{(car)}_{5,6} = R^{(car)}_{5,7} = 0.5$.
	
	The simulation is conducted over 400 time steps of size $h = 5 $ seconds, yielding a time horizon of $2000$ seconds. Furthermore, we assume an exogenous inflow given by $\lambda_1^{(car)}(t) = 1620$ veh/h and $\lambda_1^{(truck)}(t) = 180$ veh/h, for any $t \leq 250$. After that, the inflow increase to $\lambda_1^{(car)}(t) = 2700$ veh/h and $\lambda_1^{(truck)}(t) = 300$ veh/h, for any $t \in [250,750]$ and identically zero thereafter. Moreover, at time $t = 0$, we assume the network to be empty, i.e. $x_i^{(car)} = x_i^{(truck)} = 0$ for every $i$ in $\mc{E}$. 

    Once the setting is clear, we shall then use ramp meters and variable speed limits for the first $1250$ seconds to attempt to reduce congestion phenomena, focusing mainly on cells before the branch. In particular, we shall set the control parameter of cars to 
\[
\alpha^{(car)}(t) =
\begin{bmatrix}
1 & 0.94 & 0.78 & 0.6 & 1 & 1 & 1 & 0.99 & 0.99 & 1
\end{bmatrix}, \quad 0 \le t \le 625
\]
\[
\alpha^{(car)}(t) =
\begin{bmatrix}
0.83 & 0.63 & 0.50 & 0.30 & 1 & 1 & 1 & 1 & 1 & 1
\end{bmatrix}, \quad 625 < t \le 1250\;.
\]
On the other hand, we shall also control trucks by assigning the following control parameters
\ben
\ba{rcl}
\alpha^{(truck)}(t) & = & \left[0.18 \;\; 0.18 \;\; 0.19 \;\; 0.20 \;\; 0.41 \;\; 1 \;\; 1 \;\; 0.99 \;\; 0.99 \;\; 1\right]\;, \;\;  0 \leq t \leq 625 \\
\ea 
\een
\ben
\ba{rcl}
\alpha^{(truck)}(t) & = & \left[ 
0.58 \;\; 0.37 \;\; 0.14 \;\; 0.03 \;\; 0.25 \;\; 1 \;\; 1 \;\; 1 \;\; 1 \;\; 1
\right]\;, \;\; 625 < t \le 1250\;.
\ea
\een
After $1250$ seconds we shall stop any control action, i.e. $\alpha_i^{(car)} = \alpha_i^{(truck)} = 1$ for every $i$ in $\mc{E}$. We show the results in Figures \ref{fig:unc} and \ref{fig:unopt_cont}, where it is possible to notice that in both scenarios a congestion starts in cell $4$ and it is propagated back to previous cells. However, lowering vehicle inflow by ramp metering on the first cell and reducing the speed limits of cells before the branch, leads to a reduction in the congestion levels, as shown in Table \ref{tab:mass}, where it is seen that the total traffic volume lowers by $5\%$ with respect to the uncontrolled dynamics. Moreover, the traffic volume over time can be seen in Figure~\ref{fig:unop_contr_mass} where this heuristic control manages to achieve a lower value until $t = 1400 $ seconds.
\begin{table}
\centering
\begin{tabular}{lcc}
Total Traffic Volume & Uncontrolled & Heuristic \\ \hline
Onramps   & 6693   & 8104   \\ 
Cells   & 36310   & 31987  \\ 
Offramps & 1219 &  1527\\ \hline
Total & 44222 & 41618 \\ \hline
\end{tabular}
\label{tab:mass}
\caption{Total traffic volume divided by cell's type comparing uncontrolled and heuristic dynamics.}
\end{table}
	\begin{figure}
		\centering
		\begin{minipage}{0.45\textwidth}
			\centering
			\includegraphics[width=\linewidth]{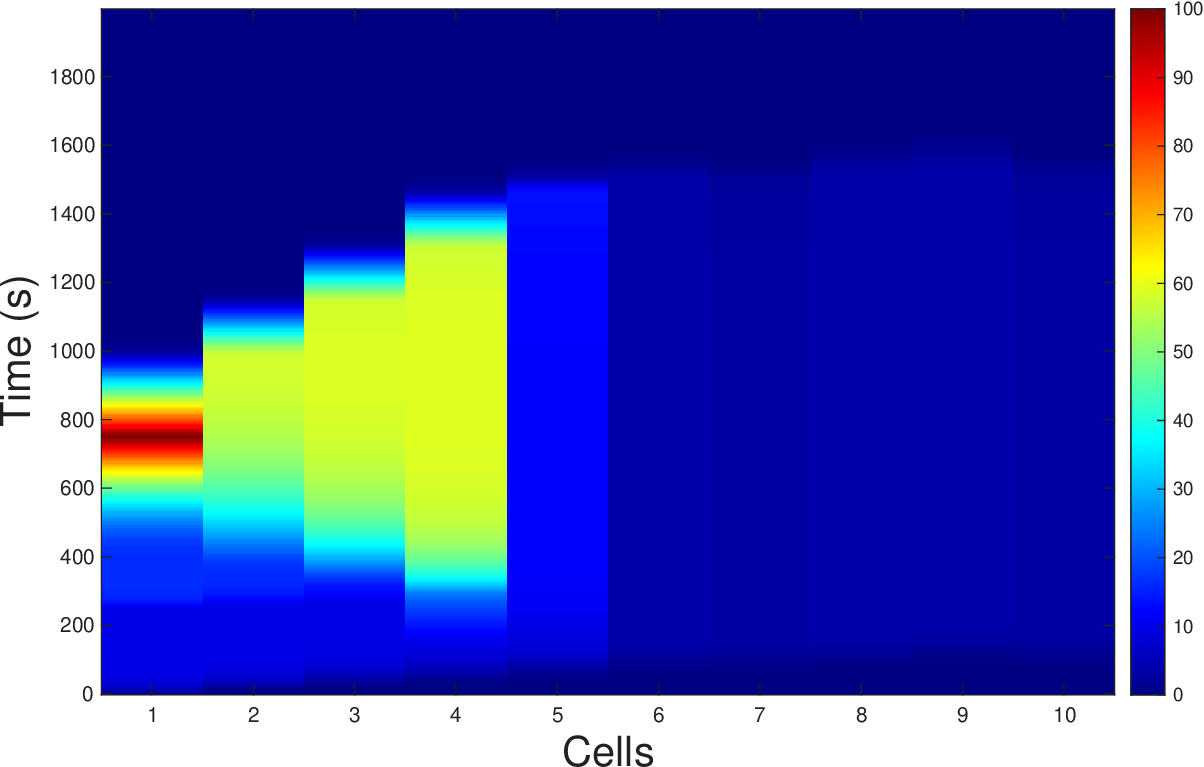}
			\caption{Mass plot of uncontrolled dynamics}
			\label{fig:unc}
		\end{minipage}\hfill
		\begin{minipage}{0.45\textwidth}
			\centering
			\includegraphics[width=\linewidth]{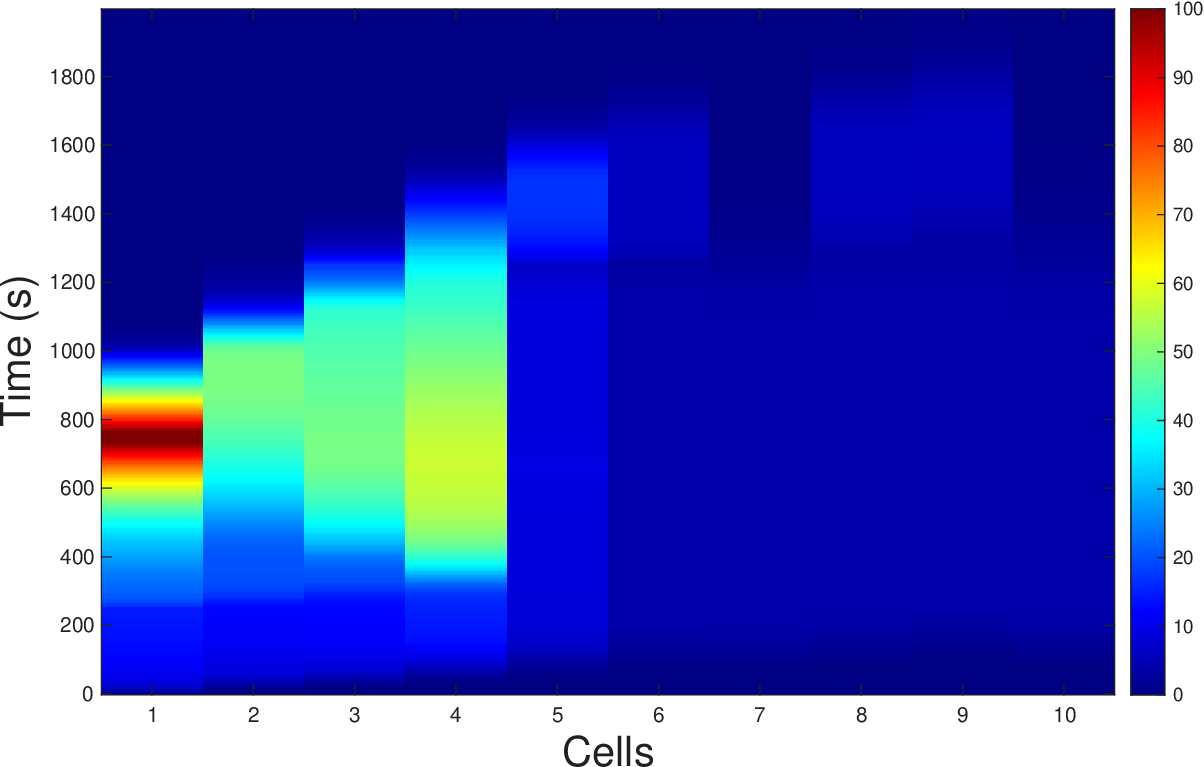}
			\caption{Mass plot with heuristic control}
			\label{fig:unopt_cont}
		\end{minipage}
	\end{figure}
	\begin{figure}
	    \centering
	    \includegraphics[width=0.5\linewidth]{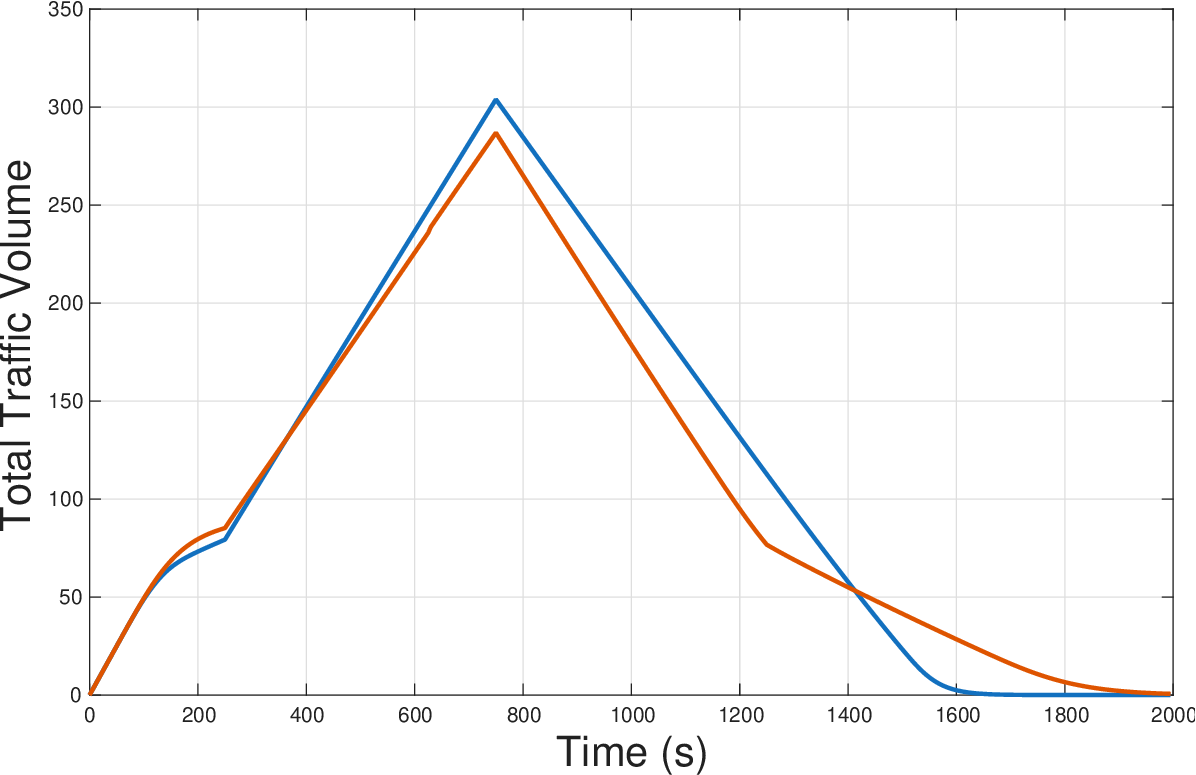}
	    \caption{Total traffic volume over time: in blue uncontrolled dynamics and in red heuristic control}
	    \label{fig:unop_contr_mass}
	\end{figure}

\end{contexample}

Once the model and its dynamics have been established, a natural question is whether it can be leveraged for traffic flow optimization via optimal control techniques. In particular, we shall consider the Freeway Network Control (FNC) problem, i.e., when the control is achieved through variable speed limits and ramp metering.

\subsection{Optimal Multi-Commodity Freeway Network Control Problem}
The FNC problem treats routing as an exogenous input and focuses solely on optimizing the control parameters $\alpha_i^{(k)}$ and it can be formulated as an open-loop optimal control problem, where the objective is to minimize the time-integral over cost $\varphi(x,z)$ which is a function of the traffic volume and the cell outflow, over a fixed time interval $\left[0,T\right]$, with $T$ denoting a finite time horizon.

\begin{assumption}\label{ass:cost}
	The cost function $\varphi(x,z)$ is convex and differentiable in $(x,z)$ and such that
	\ben
	\frac{\partial \varphi}{\partial x}(x,z) \geq 0\,, \quad \frac{\partial \varphi}{\partial z}(x,z) \leq 0\,, \quad
	\varphi(0,0) = 0\;,
	\een
	i.e., non-decreasing in $x$, non-increasing in $z$, and equal to zero in $(x,z) = (0,0)$. 
\end{assumption}
\begin{example}
	Examples of common cost functions that satisfy Assumption~\ref{ass:cost} are:
    
    \begin{enumerate}
		\item[(i)] Total Travel Time (TTT):
		\ben
		\varphi(x,z) = \sum\limits_{i \in \mc{E}}{\varphi_i(x_i,z_i)} = \sum\limits_{i \in \mc{E}}{\sum\limits_{k \in \mc{K}}{x_i^{(k)}}} \;.
		\een
		where $\int_0^T{\varphi(x,z)} \d t$ then has the interpretation of time spent by all vehicles on every cell during the time span $\left[0,T\right]$;
		\item[(ii)] Total Travel Distance (TTD):
		\ben
		\varphi(x,z) = \sum\limits_{i \in \mc{E}}{\varphi(x_i,z_i)} = \sum\limits_{i \in \mc{E}}{-l_i\sum\limits_{k \in \mc{K}}{z_i^{(k)}}}\;,
		\een
		where $l_i$ denotes the length of cell $i$ in $\mc{E}$.
		Then, $\int_0^T{\varphi(x,z)} \d t$ can be interpreted as the distance traveled by all drivers in the interval $\left[0,T\right]$.
	\end{enumerate}
\end{example}

With the introduction of a cost function and hence a control objective, it is now possible to formally state the optimal control problem.

\begin{problem}[Continuous time MC-FNC]\label{problem1}
	\be\label{eq:FNC}
	\ba{rcll}
	& \underset{\alpha(t)}{\emph{minimize}} & \displaystyle\int_0^T{\varphi(x(t),z(t)) \d t} & \quad \\[20pt]
	& \emph{subject to} &  \eqref{eq:massCons}, \eqref{eq:zik}, \eqref{eq:fifo}, \eqref{eq:4.5} & \quad \\
	\ea
	\ee
\end{problem}


Notice that the FIFO allocation rule \eqref{eq:fifo} and \eqref{eq:4.5} make the optimization problem non-convex, as observed in \cite{Schmitt2018}.

\section{A Tight Convex Relaxation of the MC-FNC problem}
 In order to tackle the non-convexity of the problem, we introduce a relaxation to make it convex. The main idea is to change the optimization variables from $\alpha$ to the traffic volume variables $\bar x$ in $\R_+^{\mc{E}\times\mc{K}}$, and flow variables $\bar z$ in $\R_+^{\mc{A}\times\mc{K}}$ and $\bar f $ in $\R_+^{\mc{A}\times\mc{K}}$. Moreover, these variables need to satisfy the supply and demand constraints, which in turn stem from the fundamental diagram of traffic flows. Hence, our newly introduced variables have to satisfy:
\begin{enumerate}
	\item[(i)] demand constraints
	\be\label{eq:dmnd}
		\bar{z}_i^{(k)} \leq d_i^{(k)}(\bar{x}_i^{(k)}) \, , \qquad \forall i \in \mc{E}\,, \quad \forall k \in \mc{K}\; ;
	\ee
	\item[(ii)]  supply constraints
	\be\label{eq:supp}
		\sum\limits_{k\in\mc{K}}{\sum\limits_{j\in\mc{E}}{\bar{f}_{ji}^{(k)}(x)}} \leq s_i(\sum\limits_{k\in\mc{K}}{w_i^{(k)}\bar{x}_i^{(k)}})\,, \qquad \forall i \in \mc{E}\;.
	\ee
\end{enumerate}

In order to formulate the relaxation of the MC-FNC problem, the cell-to-cell flow $\bar f_{ij}^{k}$ must be equal to the product between the turning ratio $R_{ij}^{(k)}$ and the cell outflow $\bar z_i^{(k)}$. Formally,
\be\label{eq:frz}
	\bar f_{ij}^{k} = R_{ij}^{(k)}\bar z_i^{(k)} \,, \qquad \forall (i,j) \in \mc{A} \,, \quad \forall k \in \mc{K}\;.
\ee
With the newly introduced variables and constraints, the optimal control problem reads as follows:
\begin{problem}[Relaxation of the continuous time MC-FNC]\label{problem2}
\be\label{eq:FNC_r}
	\ba{rcll}
	 & \underset{\bar x(t),\bar z(t),\bar f(t)}{\emph{minimize}}  & \displaystyle\int_0^T{\varphi(\bar x(t),\bar z(t)) \d t} \quad & \\[20pt]
	& \emph{subject to} & \dot{\bar{x}}^{(k)}(t) = \lambda^{(k)}(t)  + \bar{f}^{(k)}(x(t)) - \bar{z}^{(k)}(x(t) ) \,,  & \\[10pt]
	& & \bar{z}_i^{(k)}(t)  \leq d_i^{(k)}(\bar{x}_i^{(k)}(t) )\,,   & \\[10pt]
	& & \sum\limits_{k\in\mc{K}}{\sum\limits_{j\in\mc{E}}{\bar{f}_{ji}^{(k)}(x(t) )}} \leq s_i(\sum\limits_{k\in\mc{K}}{w_i^{(k)}\bar{x}_i^{(k)}(t) })\,,\\[10pt]
	& & \bar f_{ij}^{k}(x(t)) = R_{ij}^{(k)}\bar z_i^{(k)}(x(t) )\,, & \\[10pt]
	& & x^{(k)}(0) = x^{(k),0} \,, &\\[10pt]
	& & \forall i \in \mc{E},\quad \forall k \in \mc{K},\quad 
 \forall (i,j) \in \mc{A}\;.
	\ea
\ee
\end{problem}

The following simple result ensures that  Problem \ref{problem2} is indeed a convex relaxation of Problem \ref{problem1}, i.e., that Problem \ref{problem2} is a convex optimal control problem and that every feasible solution of Problem \ref{problem1} is also feasible for Problem \ref{problem2}. 

\begin{proposition}
   Problem \ref{problem2} is a convex relaxation of Problem \ref{problem1}. 
\end{proposition}
\begin{proof}
	First, we prove that every feasible solution of \eqref{eq:FNC} is a feasible solution of \eqref{eq:FNC_r}. In fact, for every choice of control parameters $\alpha$ in $[0,1]$, the dynamic of the system inevitably satisfies  \eqref{eq:dmnd}, \eqref{eq:supp}, and \eqref{eq:frz}, implying it is feasible for the relaxed problem.
	
	We are now left to prove the convexity of \eqref{eq:FNC_r} by noticing that this would mean that if both 
	$\left( \bar x^0(t), \allowbreak,\bar z^0(t),\bar f^0(t)\right)$ and $\left( \bar x^1(t),\bar z^1(t),\bar f^1(t)\right)$ satisfy the constraints, then for every $\beta$ in $\left[0,1\right]$, also $\left( \bar x^\beta(t),\bar z^\beta(t),\bar f^\beta(t)\right)$ does, where $\bar x^\beta = (1-\beta)\bar x^0 + \beta \bar x^1$, $\bar z^\beta = (1-\beta)\bar z^0 + \beta \bar z^1$ and $\bar f^\beta = (1-\beta)\bar f^0 + \beta \bar f^1$. This implies that
	\be\label{eq:cond}
	\int_0^T{\varphi(\bar x^\beta(t)) \d t} \leq (1-\beta)\int_0^T{\varphi(\bar x^0(t))\d t} + \beta\int_0^T{\varphi(\bar x^1	(t))\d t} \,.
	\ee
	Given the convex nature of the cost function $\varphi(x,z)$, inequality \eqref{eq:cond} is always satisfied. Moreover, the concavity of the demand and supply functions implies the convexity of the constraints $\eqref{eq:dmnd}$ and $\eqref{eq:supp}$. Furthermore, for every choice of the control parameters $\alpha^{(k)}$, the constraints \eqref{eq:zik} and \eqref{eq:frz} are linear equalities.
\end{proof}

It remains to establish the tightness of the above convex relaxation. Specifically, we will show that for every solution of the convex optimal control problem~\eqref{eq:FNC_r}, there exists a corresponding selection of control parameters $\alpha_i^{(k)}$ such that equations \eqref{eq:massCons},\eqref{eq:zik},\eqref{eq:fifo}, and \eqref{eq:4.5} are satisfied. To this end, consider a solution $(\bar x,\bar z,\bar f)$  and let the control parameters $\alpha$ be given by
\be\label{eq:alfa_sol}
\alpha_i^{(k)}(t) = \frac{\bar z_i^{(k)}(t)}{d_i^{(k)}(\bar x_i^{(k)}(t))}, \qquad \forall i \in \mc{E}, \quad \forall k \in \mc{K}\;,
\ee
with the convention that $\alpha_i^{(k)} = 1$ if $\bar z_i^{(k)} = d_i^{(k)}(\bar x_i^{(k)}) = 0$. 
\begin{theorem}
	Given a C-DMFN and initial traffic volumes $x^0$, then for any feasible solution $(\bar x(t),\bar z(t),\bar f(t))$ of the convex FNC \eqref{eq:FNC_r}, it is possible to set $\alpha_i^{(k)}(t)$ as defined in \eqref{eq:alfa_sol}. Then, the solution satisfies \eqref{eq:massCons},\eqref{eq:zik},\eqref{eq:fifo},\eqref{eq:4.5} and $\alpha^{(k)}$ is a solution of the original FNC problem \eqref{eq:FNC}.
\end{theorem}
\begin{proof}
	Let $(\bar x(t),\bar z(t),\bar f(t))$ be a feasible solution. Then, the demand constraints~\eqref{eq:dmnd} guarantee that the outflow is less than or equal to the demand, i.e.,  $\bar z_i^{(k)} \leq \alpha_i^{(k)}d_i^{(k)}(\bar x_i^{(k)})$. Hence, when selecting the control parameters as in \eqref{eq:alfa_sol}, the outflow of every cell corresponds exactly to the demand. Formally,
	\ben
	\bar z_i^{(k)} = \alpha_i^{(k)}d_i^{(k)}(\bar x_i^{(k)})\,, \qquad \forall i \in \mc{E}\; .
	\een
	Furthermore, it follows that
	\ben
	\bar f_{ij}^{(k)} = R_{ij}^{(k)}\bar z_i^{(k)} = R_{ij}^{(k)}\alpha_i^{(k)}d_i^{(k)}(\bar x_i^{(k)})\,,\qquad \forall (i,j) \in \mc{A}\;,
	\een
	which implies that for every cell $j$ in $\mc{E}$
	\ben
	\sum\limits_{k \in \mc{K}}{\sum\limits_{i \in \mc{E}}{R_{ij}^{(k)}\alpha_i^{(k)}d_i^{(k)}(\bar x_i^{(k)})}} = \sum\limits_{k \in \mc{K}}{\sum\limits_{(i,j)\in\mc{A}}{\bar f_{ij}^{(k)}}} \leq s_{j}\left( \sum\limits_{k \in \mc{K}}{w_j^{(k)}\bar x_j^{(k)}}\right) \; .
	\een
	Notice that the last inequality follows from the supply constraint \eqref{eq:supp} implying that there is no congestion, i.e., $\gamma_i^{F} = 1$ for every $i$ in $\mc{E}$. Hence, $\bar x = x, \bar z = z$, and $\bar f = f$. Moreover, the equations \eqref{eq:massCons},\eqref{eq:zik},\eqref{eq:fifo}, and \eqref{eq:4.5} are satisfied so that $\alpha^{(k)}$ is a  solution of the original FNC \eqref{eq:FNC}.
\end{proof}
\begin{remark}
	It is worth emphasizing that, even without the employment of ramp meters and variable speed limits, the relaxation \eqref{eq:FNC_r} of the MC-FNC is still convex. However, the employment of variable speed limits ensures the tightness of the convex relaxation.
\end{remark}

Before stating the discrete-time MC-FNC, we will make an additional assumption to ensure that a CFL-like condition is satisfied. To do so, for each cell $i$ we will make use of its length $L_i$ and its free-flow speed $v_{ff,i}^{(k)}$ for every commodity, i.e., the speed at which each commodity travels at low traffic volume conditions.
\begin{assumption}\label{ass:linearity}
        For every cell $i$ in $\mc{E}$, we have that:  
        \begin{enumerate}
            \item[(i)] the demand function of commodity $k$ in $\mc{K}$ is a linear function of the form
            $$
            d_i^{(k)}(x_i^{(k)}) = \frac{v_{ff,i}^{(k)}}{L_i}x_i^{(k)}\;;
            $$
            \item[(ii)] the supply function is an affine function of the weighted sum $\sum\limits_{k \in \mc{K}} {w_i^{(k)}x_i^{(k)}}$;
            \item[(iii)] the discretization step $h$ must satisfy the CFL condition:
	\ben
	h \leq \min\limits_{k \in \mc{K}} \frac{L_i}{v_{ff,i}^{(k)}} \; .
	\een
    \end{enumerate}
\end{assumption}
    Notice that point (iii) of Assumption $\ref{ass:linearity}$ must hold for every cell, so we can rewrite it compactly as
    $$
    h \leq \min\limits_{i \in \mc{E}, k \in \mc{K}} \frac{L_i}{v_{ff,i}^{(k)}} \; .
    $$
    We are now ready to introduce the MC-FNC problem in discrete time. 
\begin{problem}[Discrete time relaxation of the MC-FNC problem]
	\ben
	\ba{rcll}
		& \underset{\bar x(t),\bar z(t),\bar f(t)}{\emph{minimize}} & \sum\limits_{t=0}^N{}\sum\limits_{k\in\mc{K}}{}\sum\limits_{i\in\mc{E}}{\varphi_i^{(k)}(\bar x_i^{(k)})} & \\
		& \emph{subject to}   & \bar x_i^{(k)}(t+1) = \bar x_i^{(k)}(t) +  h(\lambda_i^{(k)}(t)  + \sum\limits_{j\in\mc{E}}{\bar f_{ji}^{(k)}}(t) - \bar z_i^{(k)}(t))\,,&\\
        & & 0 \leq \bar z_i^{(k)}(t) \leq d_i^{(k)}(\bar x_i^{(k)}(t)),\quad & \\
        & & \sum\limits_{k\in\mc{K}}{\sum\limits_{j\in\mc{E}}}{\bar f_{ji}^{(k)}}(t) \leq s_i(\sum\limits_{k\in\mc{K}}{w_i^{(k)}\bar x_i^{(k)}(t)}),& \\
		& & \bar f_{ij}^{(k)}(t) = R_{ij}^{(k)}(t)\bar z_i^{(k)}(t),&  \\
		& & \bar f_{ij}^{(k)}(t) \geq 0,&\\
        & & \bar x_i^{(k)}(0) = \bar x_i^{(k),0}, &\\[10pt]
	& & \forall i \in \mc{E},\quad \forall k \in \mc{K},\quad 
 \forall (i,j) \in \mc{A}\;.
	\ea
	\een
\end{problem}
In order to show the effectiveness of the control action we now present an example with a rather simple network topology.
\setcounter{example}{1}
\begin{contexample}
Let us once again continue our example on the network in Figure \ref{fig:grafo_esempio2}. In particular, recall the demand and supply functions, the turning ratios, the exogenous inflows and the time horizon. As a cost function, the sum of the traffic volume $\varphi(x) = \sum\nolimits_{k \in \mc{K}} \sum\nolimits_{i \in \mc{E}} x_i^{(k)}$ is chosen. 
The optimization is performed via the Matlab package CVX and the results are shown in Figures \ref{fig:CVX_nonopt} and \ref{fig:CVX_opt}, where it is possible to notice that with a proper combination of ramp metering and variable speed limits, the congestion is greatly reduced. In particular, in Table \ref{tab:mass_controlled} it is possible to notice that the optimal control achieves an improvement of $41.5\%$ with respect to the uncontrolled dynamics. As before, in Figure \ref{fig:contr_mass_over_time} we plot the total traffic volume over time where the optimal control dynamics always achieves a lower value than the uncontrolled one.

\begin{table}
\centering
\begin{tabular}{lcc}
Total Traffic Volume & Uncontrolled & Optimal Control \\ \hline
Onramps   & 6693   & 5123   \\ 
Cells   & 36310   & 20296  \\ 
Offramps & 1219 &  609\\ \hline
Total & 44222 & 26965\\ \hline
\end{tabular}
\label{tab:mass_controlled}
\caption{Total traffic volume divided by cell's type comparing uncontrolled and optimal control.}
\end{table}

\begin{figure}
	\centering
	\begin{minipage}{0.45\textwidth}
		\centering
		\includegraphics[width=\linewidth]{Pictures/density_plot_uncontrolled.eps}
		\caption{Mass plot of uncontrolled dynamics }
		\label{fig:CVX_nonopt}
	\end{minipage}\hfill
	\begin{minipage}{0.45\textwidth}
		\centering
		\includegraphics[width=\linewidth]{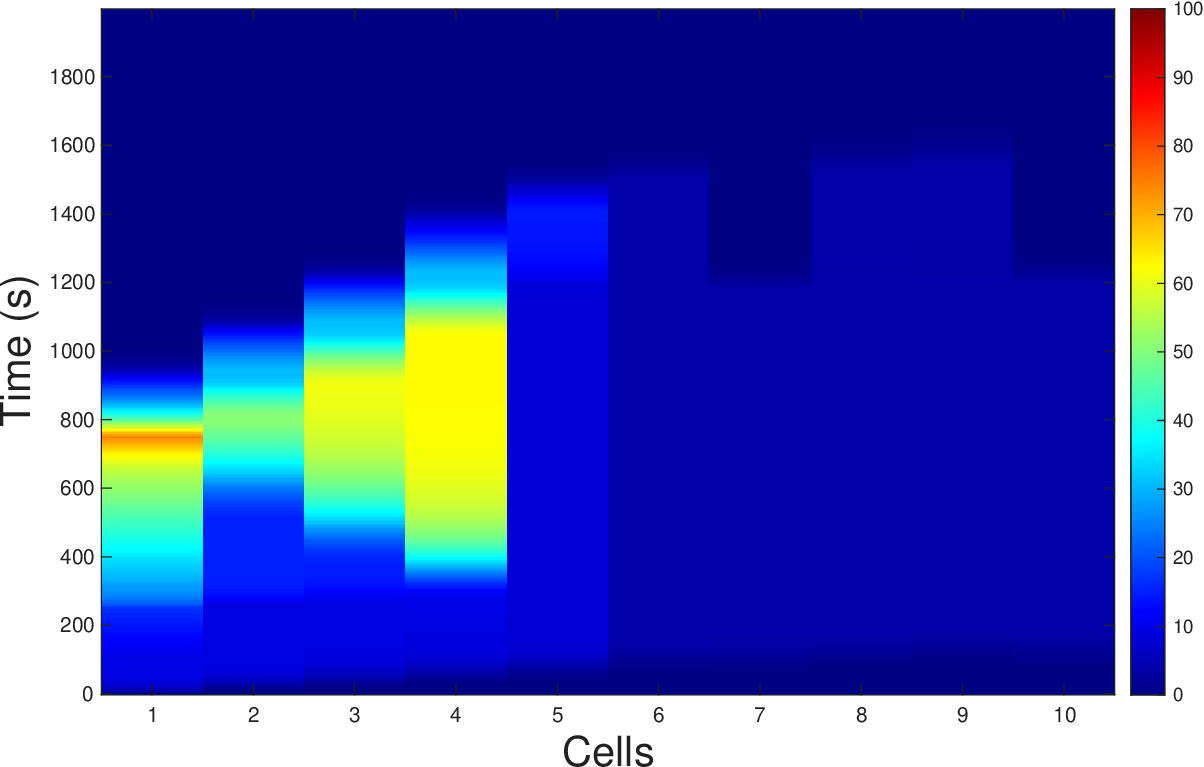}
		\caption{Mass plot of optimal control dynamics}
		\label{fig:CVX_opt}
	\end{minipage}
\end{figure}

\begin{figure}
    \centering
    \includegraphics[width=0.5\linewidth]{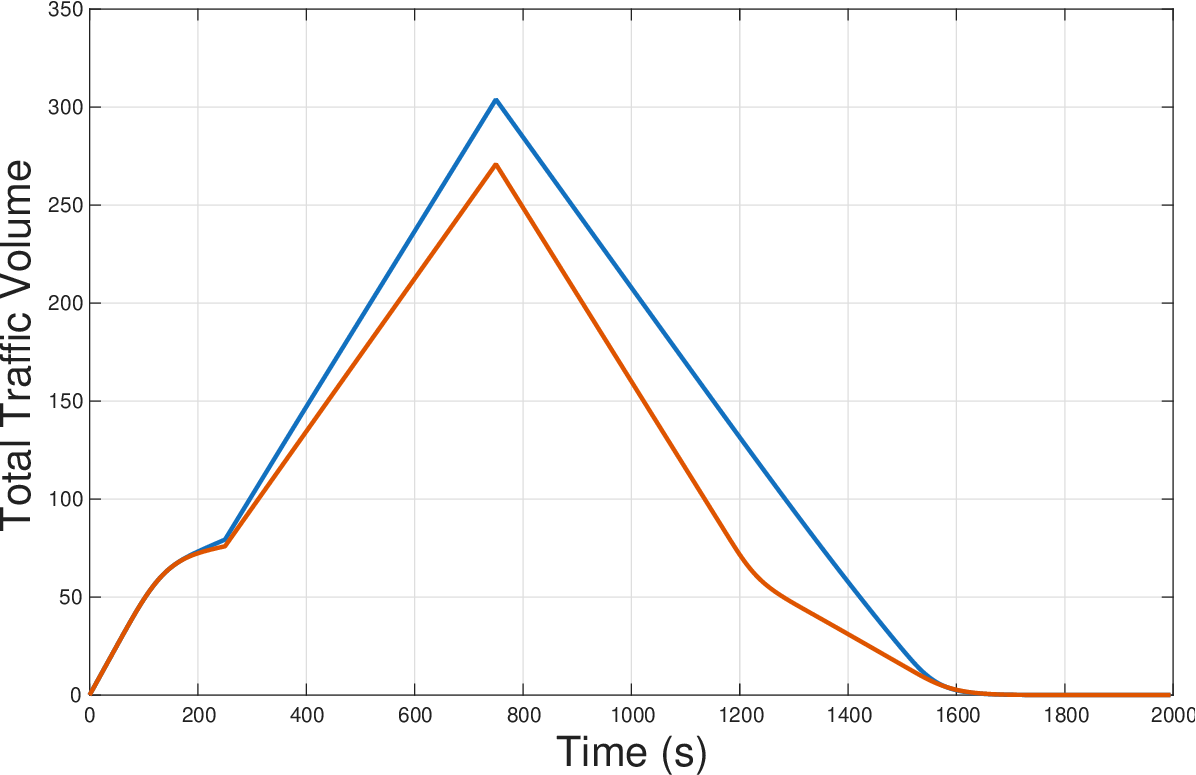}
    \caption{Total traffic volume over time: in blue uncontrolled dynamics and in red optimal control dynamics.}
    \label{fig:contr_mass_over_time}
\end{figure}
\end{contexample}

\section{Case Study: Caltrans Performance Measurement System (PeMS) Database }
To test the real effectiveness of the proposed optimal control scheme, a scenario using data from the \textit{Caltrans Performance Measurement System (PeMS)}\footnote{https://pems.dot.ca.gov/} database was conducted, focusing on a subset of roads around Los Angeles, as shown in Figure \ref{fig:grafo_pems}. Let us call the network in the figure $\tilde{\mc{G}} = (\tilde{\mc{V}},\tilde{\mc{E}})$ and we shall call every link $i$ in $\tilde{\mc{E}}$ a \textit{road}. Then, each road $i$ in $\tilde{\mc{E}}$ is split into cells $i_1, \dots, i_n$ of length 2 miles, with an ordinary junction between consecutive cells. To each cell, an onramp and an offramp are then associated, obtained by collapsing together all the actual ramps in that 2-mile section of the road. In order to be consistent with previous notation, we will identify $\mc{E}$ as the set of all cells, onramps, and offramps and $\mc{V}$ as the set of junctions between cells and between roads. In order to model our different commodities, we extracted sensor data regarding both cars and trucks separately. By using the measured data, the demand functions, supply functions, and routing matrices of each commodity were calibrated, as explained in more detail in~\ref{app:pems}. The initial conditions, $x^{(car),0}$ and $x^{(truck),0}$, are set by considering the flow measurement taken from the first sensor in each cell on February 8, 2012 at 06:00 AM, for cars and trucks separately.

\begin{figure}
	\centering
	\begin{minipage}{0.45\textwidth}
		\centering
		\includegraphics[width=1\linewidth]{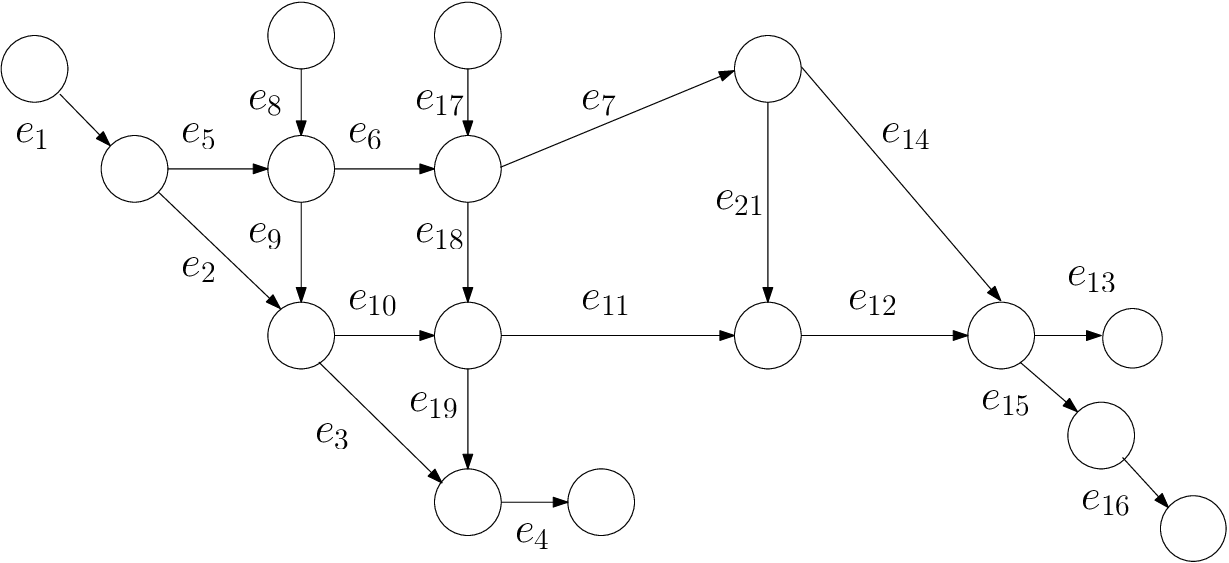}
	\end{minipage}\hfill
	\begin{minipage}{0.45\textwidth}
		\centering
		\includegraphics[width=1\linewidth]{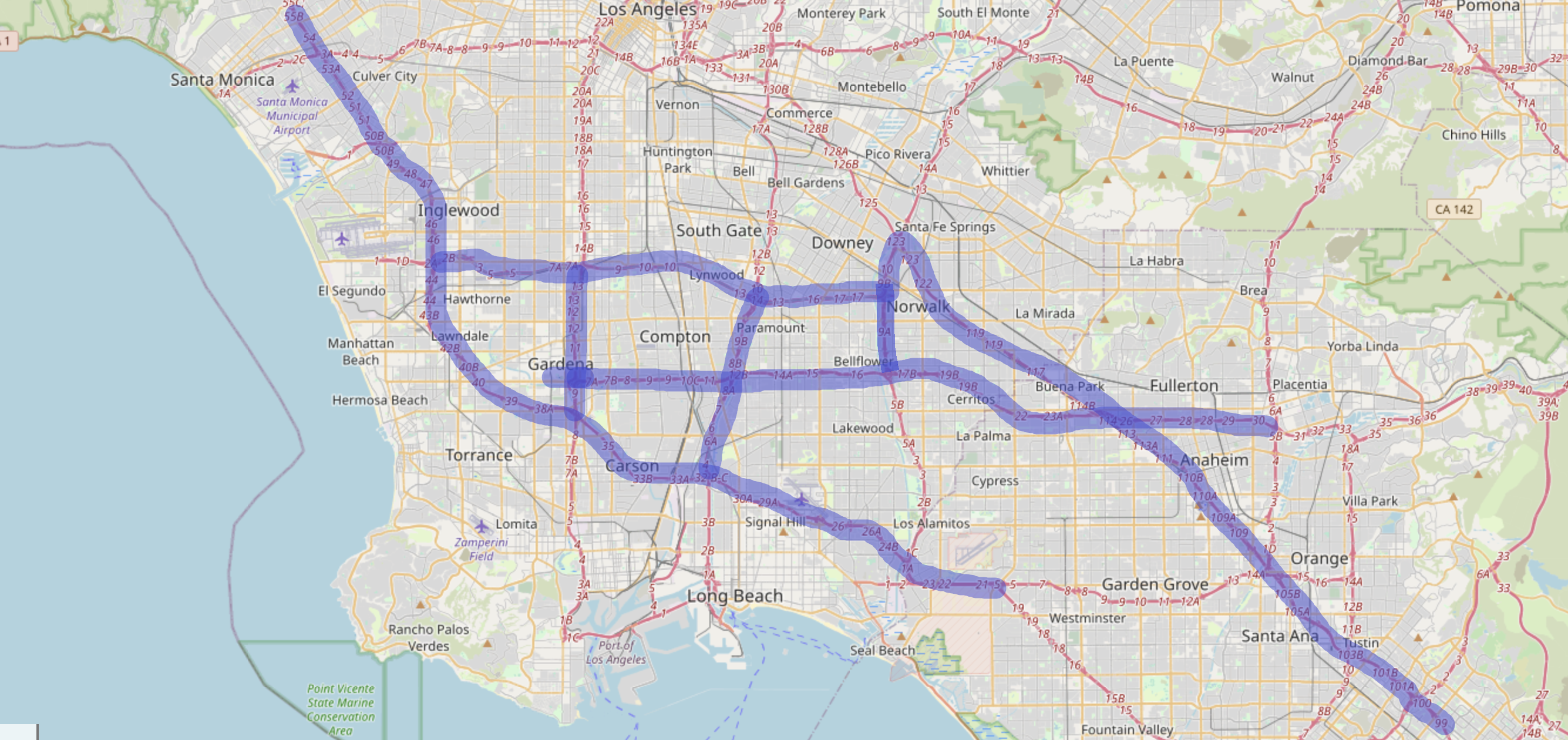}
	\end{minipage}
	\caption{On the left: graph representation of a part of the freeway network. Each link is then sectioned into 2-mile cells and also is equipped with an onramp and offramp, by collapsing together all the ramps in that section. On the right: A section of the Los Angeles map with the corresponding network highlighted.}
	\label{fig:grafo_pems}
\end{figure}

We simulate the system over a horizon $T = 4$ hours, choosing a discretization parameter $h = 72$ seconds. The exogenous inflow on each onramp $r$ in $\mc{R}$ is changed every hour to simulate different traffic conditions. In particular, for each cell $i$ in $\mc{E}$ we compute the inflow of its onramp by taking actual flow measurements on all onramps on that 2-mile section and summing them together. Moreover, we shall consider as cost functions the total travel time, i.e.,

 $$\varphi(x)_{TTT} =\sum\limits_{t=0}^{k}{} \sum\limits_{i \in \mc{E}} {} x_i^{(car)}(t) + x_i^{(truck)}(t) \;, $$

\subsection{Results of the MC-FNC Problem with PeMS Data}
In this section, we present our results by solving the MC-FNC problem to improve the total travel time with respect to the uncontrolled scenario. In particular, we shall first compute the optimal ramp metering and variable speed limits for both cars and trucks separately, to achieve the optimal control. Then, we shall assume that only cars are controlled, to understand if partial control can still improve the system's performance.

The cost function values are presented in Table \ref{tab:total_traffic_volume_optimal}, where it is clear that the optimal control achieves the best performance. In particular, it manages to improve the total travel time by $16.24\%$ with respect to the uncontrolled scenario. Moreover, applying variable speed limits and ramp metering only to cars managed to improve the system's performance. In particular, it improves the total travel time by $5.1\%$ with respect to the uncontrolled case. The total traffic volume over time is shown in Figures \ref{fig:pems_vol}, where it is possible to notice that the optimal control and the partial control of cars always manage to achieve a lower value with respect to the uncontrolled scenario. Moreover, in Figures \ref{fig:pems_speed} and \ref{fig:pems_speed_on}, it is noticeable that, by reducing the speed limits of vehicles in the mainline or applying ramp metering on the onramps, it is possible to increase the vehicle average speed.

\begin{table}
\centering
\begin{tabular}{lccc}
TTT  & Uncontrolled & Optimal Control & Partial Control Cars \\ \hline
Onramps   & 3007125   & 5819032 & 5892485\\ 
Cells   & 154629   & 180389 & 161441\\ 
Offramps & 7226348 &  2701341 & 3801783\\ \hline
Total & 10388103 & 8700763 & 9855709\\ \hline
\end{tabular}
\label{tab:total_traffic_volume_optimal}
\caption{Total traffic volume divided by cell's type comparing uncontrolled, optimal control, and partial control.}
\end{table}

\begin{figure}
	\centering
	\includegraphics[width=0.5\linewidth]{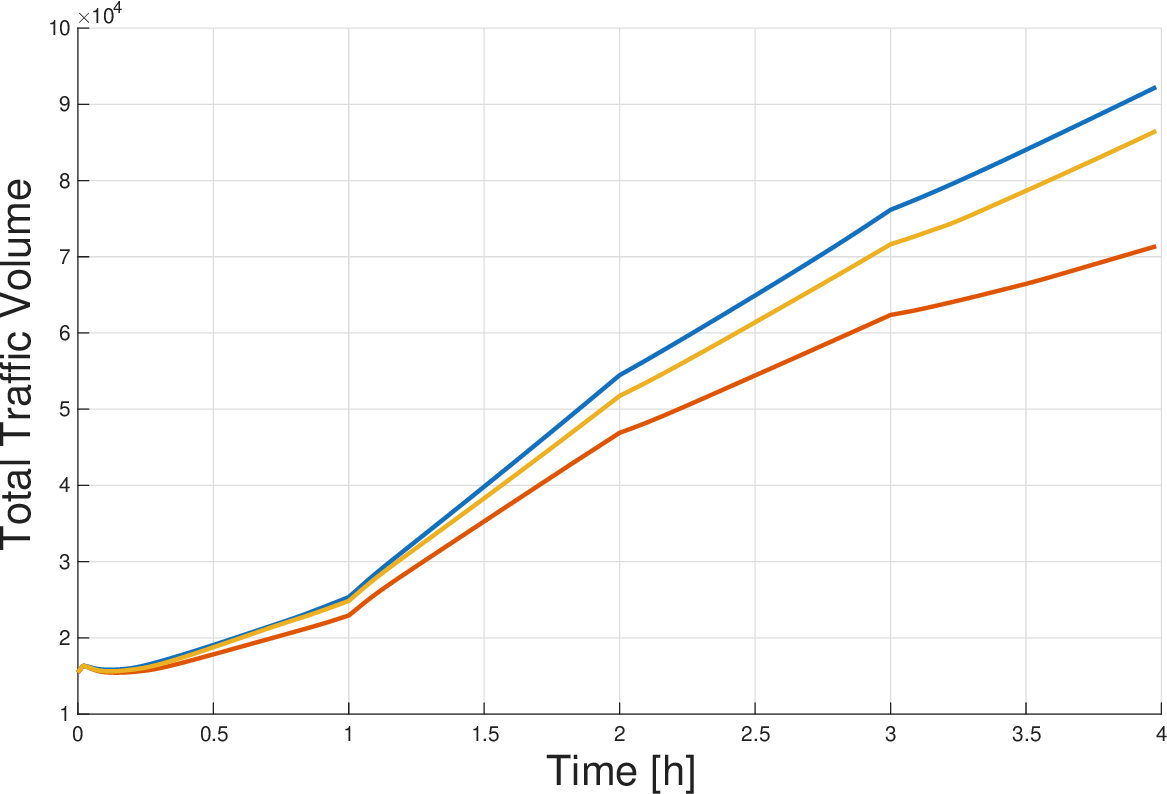}
	\caption{Total Traffic Volume comparison: in blue the uncontrolled dynamics, in red the optimal control dynamics, and in yellow the partial control dynamics.}
	\label{fig:pems_vol}
\end{figure}
\begin{figure}
\centering
\begin{minipage}{0.48\linewidth}
	\includegraphics[width=\linewidth]{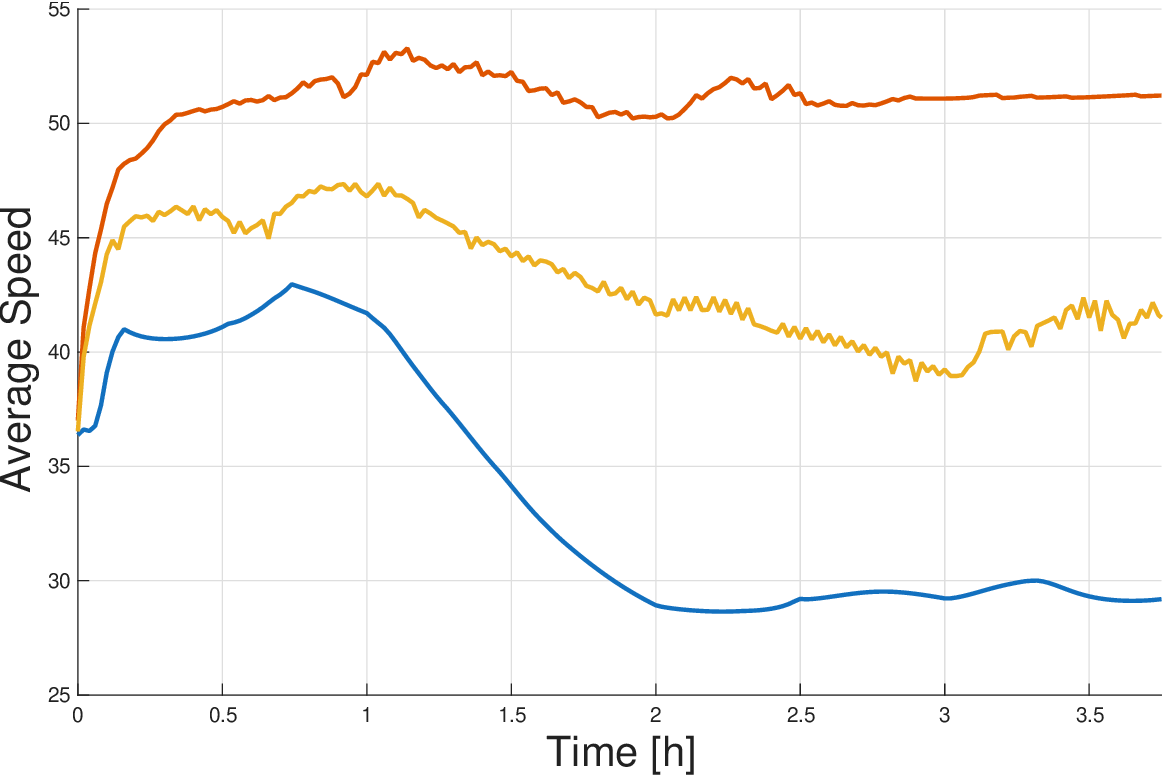}
\end{minipage}
\hfill
\begin{minipage}{0.48\linewidth}
	\includegraphics[width=\linewidth]{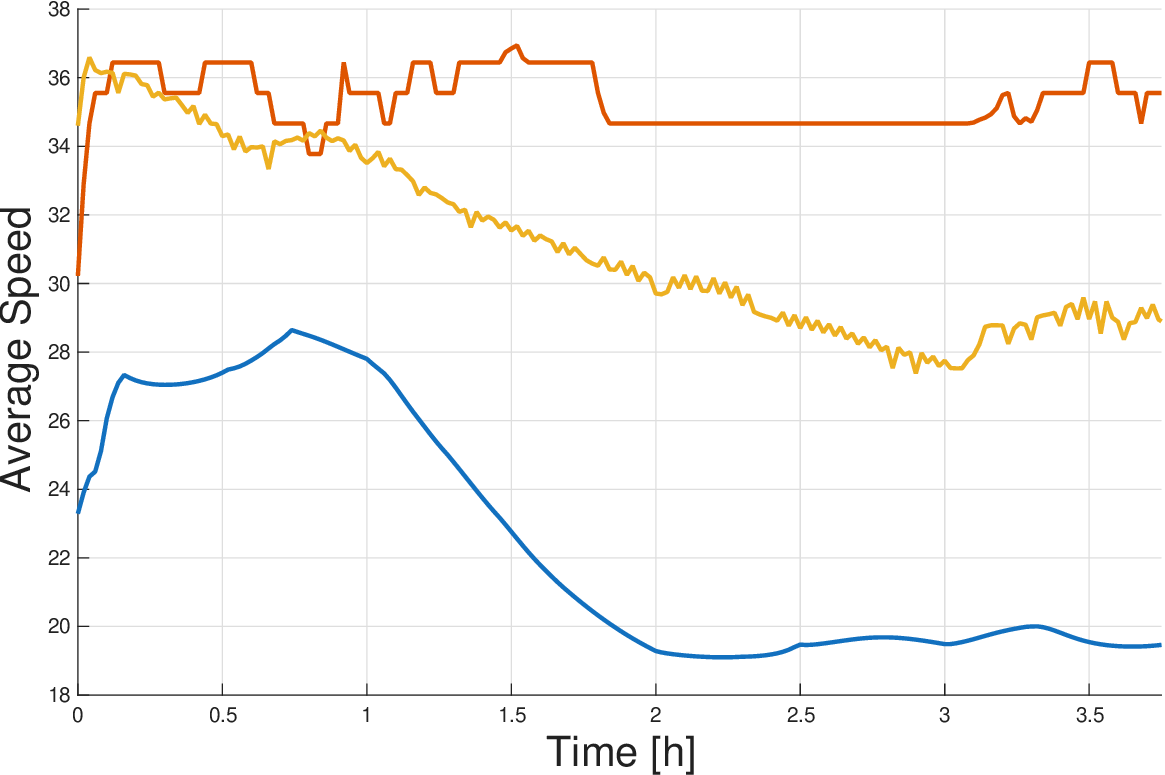}
\end{minipage}
\caption{Average Speed of cars (left) and trucks (right) in the mainline under different scenarios: in blue the uncontrolled dynamics, in red the optimal control one, and in yellow the partial control of cars.}
\label{fig:pems_speed}
\end{figure}
\begin{figure}
\centering
\begin{minipage}{0.48\linewidth}
	\includegraphics[width=\linewidth]{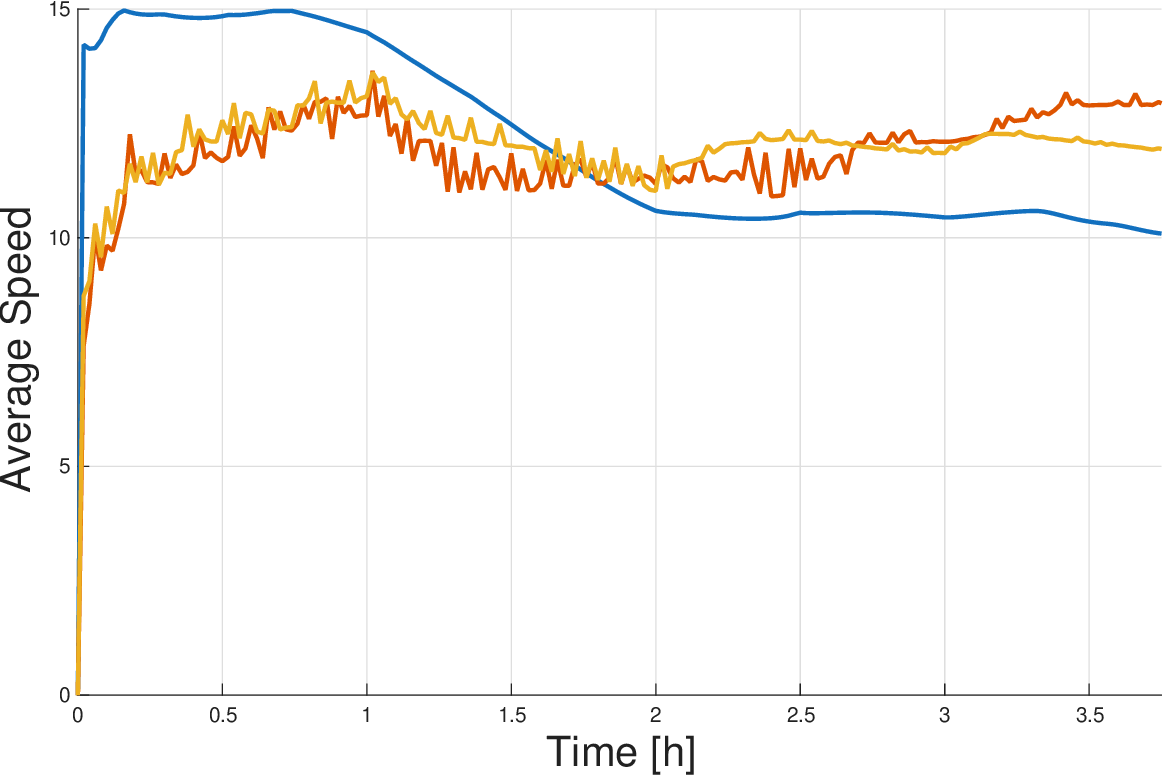}
\end{minipage}
\hfill
\begin{minipage}{0.48\linewidth}
	\includegraphics[width=\linewidth]{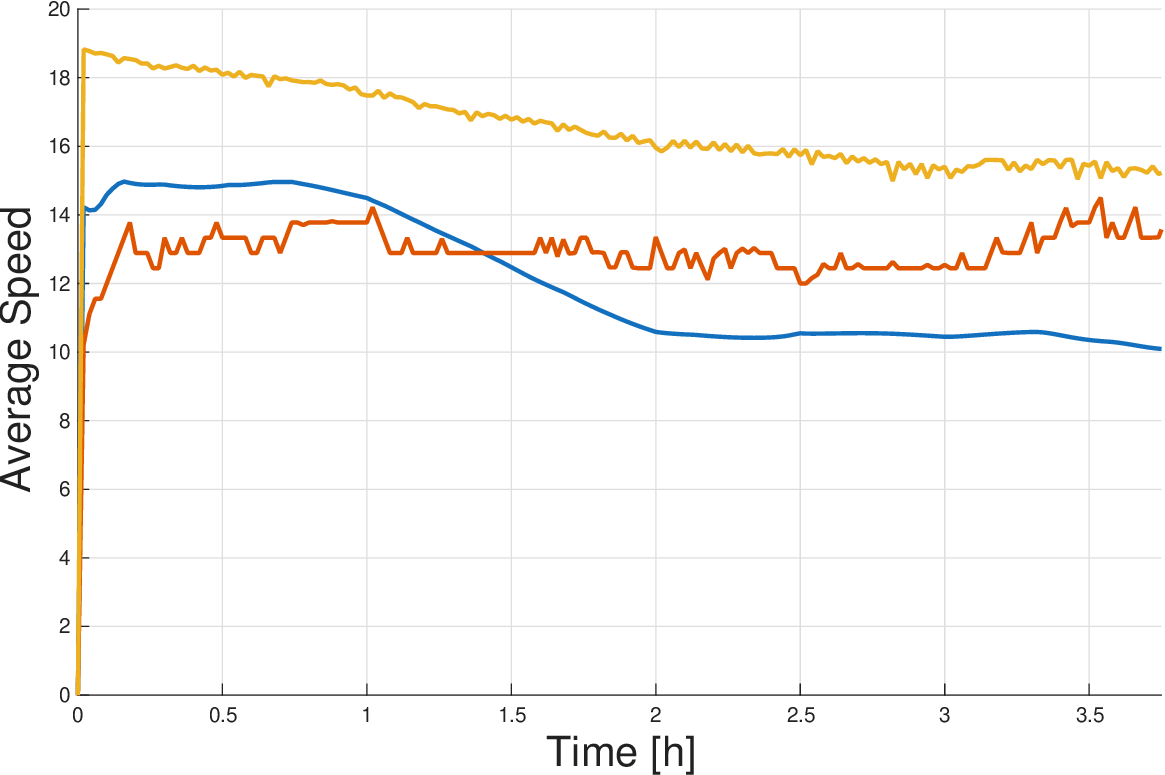}
\end{minipage}
\caption{Average Speed of cars (left) and trucks (right) in onramps under different scenarios: in blue the uncontrolled dynamics, in red the optimal control one, and in yellow the partial control of cars.}
\label{fig:pems_speed_on}
\end{figure}

\subsection{Single-commodity Comparison}
As a last study, we will now investigate whether controlling different commodities separately will actually improve performance with respect to determining the control action based on a single-commodity model. In particular, consider the case where the multi-commodity nature of the model is not known to the controller, so control cannot be applied to individual vehicle classes but only to the aggregate system. To this end, we shall consider each vehicle in the network to be a car, and the routing matrix is derived again from the PeMS data. We then proceed to compute the optimal variable speed limits by solving the single-commodity optimization problem and then apply them to the multi-commodity model. The results in Figure~\ref{fig:pems_single} show that the single-commodity control achieves poor performances, even worse than the uncontrolled multi-commodity dynamics, which emphasizes the importance of studying multi-commodity optimal control.

\begin{figure}
	\centering
		\includegraphics[width=0.8\linewidth]{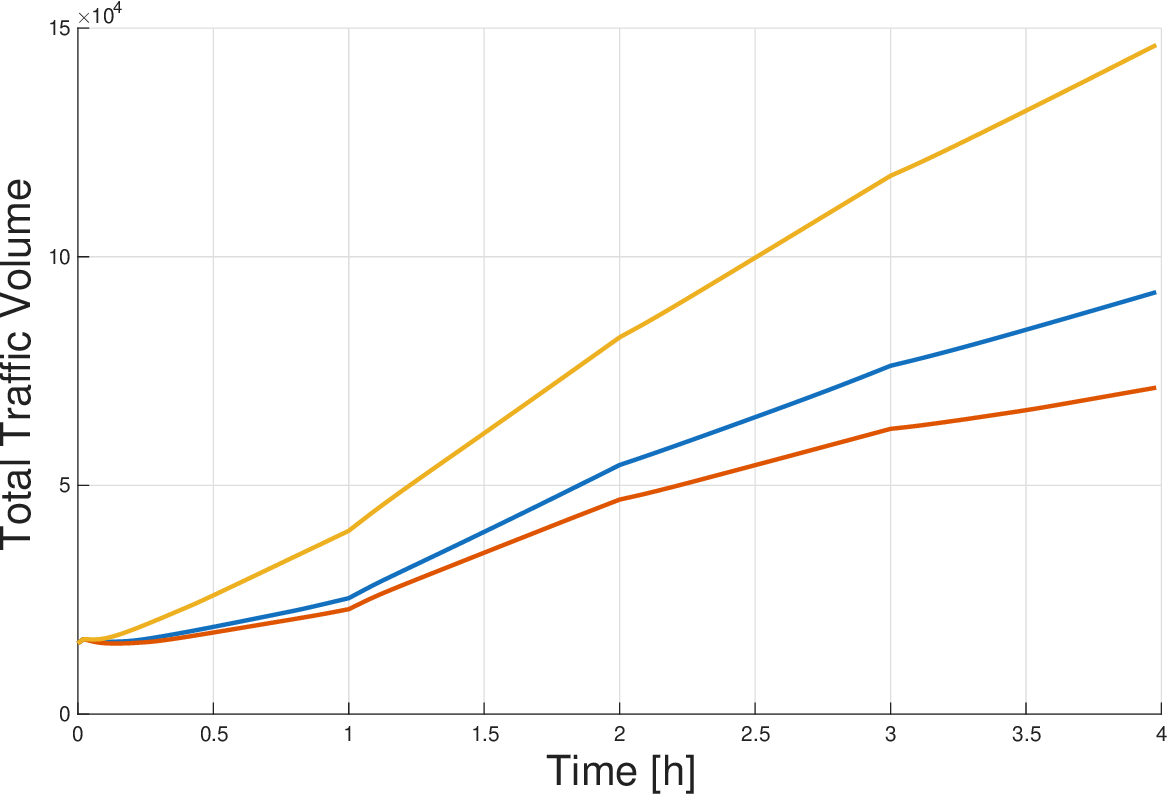}
		\caption{Total Traffic Volume comparison: in blue the uncontrolled dynamics, in red the optimal control one, and in yellow the single-commodity control dynamics.}
		\label{fig:pems_single}
\end{figure}

\section{Conclusion}
	We consider a dynamical multi-commodity flow network model based on the Cell Transmission Model. In this setting, we tackle the MC-FNC problem, an optimal control problem in which the control action is achieved through variable speed limits and ramp metering. In particular, by generalizing single-commodity research, we study a relaxation of this problem in which the total outflow of each commodity from each cell is bounded by a concave demand function, whereas the total inflow to a cell is bounded by a supply function shared by all commodities. Moreover, the design of commodity-specific variable speed limits and ramp metering yields convex formulations. Hence, the MC-FNC problem can be solved via standard tools, e.g., \textit{CVX}, as emphasized by both our synthetic and real examples.  At last, we compare single and multi-commodity control to show the benefits that come from controlling different classes of vehicles separately. Future research should focus on including the route guidance in the set of control actions and developing a distributed algorithm capable of solving the problem.

\appendix
\section{PeMS Data Calibration}\label{app:pems}
		The information and structure of the network can be seen in Table \ref{tab:1} and Figure \ref{fig:grafo_pems}.
		\begin{table}
			\centering
			\begin{tabular}{cccccc}
				Cell & Freeway & PM start & PM end\\
                \hline

				$e_1$& I405-S & 52.93 & 45.14\\
				$e_2$& I405-S & 44.37 & 37.08\\
				$e_3$& I405-S & 36.34 & 31.82\\
				$e_4$& I405-S & 30.55 & 20.65\\
				$e_5$& I105-E & 2.50 & 7.20\\
				$e_6$& I105-E &7.56 & 13.20\\
				$e_7$& I105-E & 13.86 & 17.30\\
				$e_8$& I110-S & 15.22 & 13.50\\
				$e_9$& I110-S & 13.33 & 9.93\\
				$e_{10}$& SR91-E & 0.56 & 5.4\\
				$e_{11}$& SR91-E & 6.20 & 10.22\\
				$e_{12}$& SR91-E & 11.37 & 18.14\\
				$e_{13}$& SR91-E & 19.17 & 23.87\\
				$e_{14}$& I5-S & 123.21 & 114.71\\
				$e_{15}$& I5-S & 113.99 & 107.25\\
				$e_{16}$& I5-S & 105.71 & 95.25\\
				$e_{17}$& I710-S & 12.52 & 10.50\\
				$e_{18}$& I710-S & 10.29 & 8.33\\
                $e_{19}$& I710-S & 7.54 & 4.29\\
                $e_{20}$& I605-S & 9.35 & 7.25\\
				\hline
			\end{tabular}
			\caption{PeMS data}
			\label{tab:1}
		\end{table}
		
		The roads in Table \ref{tab:1} have been divided in cells of length 2 miles and only the cells containing at least one sensor have been kept. Moreover, each cell is associated with both an onramp and an offramp obtained by collapsing together the existing ramps in the 2-mile section.
        \begin{table}
\centering
\setlength{\tabcolsep}{4pt} 
\renewcommand{\arraystretch}{1.1} 
\begin{tabular}{lccc}
\hline
Cell & Free-flow speed (Cars and Trucks) (mph) & Len. (mi) & Lanes \\ \hline
Onramps  & 20  & 0.5 & 1 \\ 
Mainline  & 60 \text{ and } 40 & 2   & 6 \\ 
Offramps & 20 & 0.5 & 1 \\ \hline
\end{tabular}
\caption{Differences in cell characteristics by vehicle type.}
\label{tab:cell_differences}
\end{table}
        
        The cells' characteristics are displayed in Table \ref{tab:cell_differences}. Hence, we can compute the demand functions as
        $$
        d_i^{(car)}(x_i^{(car)}) = \begin{cases}
            30x_i^{(car)}\;, & \forall i \in \mc{E} \setminus \left( \mc{R} \cup \mc{S}\right)\;, \\
            40x_i^{(car)} \;, & \forall i \in \left( \mc{R} \cup \mc{S}\right)\;.
        \end{cases} 
        $$
        $$
        d_i^{(truck)}(x_i^{(truck)}) = \begin{cases}
            20x_i^{(truck)}\;, & \forall i \in \mc{E} \setminus \left( \mc{R} \cup \mc{S}\right)\;, \\
            40x_i^{(truck)} \;, & \forall i \in \left( \mc{R} \cup \mc{S}\right)\;.
        \end{cases} 
        $$

        We can then define the supply functions. However, we will assume that only mainline cells limit their total inflow, so that both onramps and offramps can receive any amount of flow. In particular,  the supply functions are constructed by taking the maximum capacity registered between February 8, 2014, and March 8, 2014, on each cell. The capacity is the intersection between the supply and the sum of the demand functions. So the supply is the affine function passing through this intersection and the point $s(\sum_{k \in \mc{K}}{w_i^{(k)}x_i^{(k)}}) = 0$, which is achieved when the maximum number of vehicles is on the cell, by assuming a vehicle split of $96\%$ cars and $4\%$ trucks.

        Lastly, we need to derive the routing matrices for both cars and trucks. To this end, we first compute the adjacency matrix $A$ of the network divided into 2-mile cells with onramps and offramps. Then, for each commodity, we take flow measurements $a^{(car)}(t)$ and $a^{(truck)}(t)$ on each cell every 5 minutes between $t_0 =$ 06:00 AM and $t_{end}$ = 06:59 PM on February 8, 2012. Hence, we compute the routing matrices of both commodities as
        $$
        R^{(k)}_{ij} = \frac{\sum\limits_{t = t_0}^{t_{end}}A_{ij}a^{(k)}_j(t)} {\sum\limits_{l \in \mc{E}} {\sum\limits_{t = t_0}^{t_{end}}A_{il}a^{(k)}_l(t)}}\;,
        $$
        where $\mc{E}$ is the set of all 2-mile cells, onramps and offramps.
		
         Finally, it is necessary to satisfy the CFL condition, so that the discretization step $h < {L}/{v_{ff}}$. For the mainline, we consider $v_{ff} = v_{ff,car} = 60$ mph, we achieve $h < {1}/{30}$ hours. On the other hand, for the ramps we consider $v_{ff} = 20$ mph to achieve $h < {1}/{40}$ hours. Hence, we shall consider a discretization step of $h = 0.02$ hours $= 72$ seconds.
		\bibliographystyle{elsarticle-num}
		\bibliography{references}
\end{document}